\documentclass[12pt,a4paper]{article}
\usepackage{latexsym,amssymb,amsfonts,amsmath,amsthm,nccmath,float,enumitem,setspace,bm,authblk,appendix}
\usepackage[usenames,dvipsnames]{xcolor}
\usepackage[hidelinks]{hyperref}
\usepackage[margin=2cm]{geometry}

\setlist{topsep=3pt,partopsep=0pt,itemsep=1pt,parsep=0pt}

\hypersetup{
    colorlinks,
    linkcolor={red!80!black},
    citecolor={blue!80!black},
    urlcolor={blue!80!black}
}

\newtheorem{Theorem}{Theorem}[section]

\newtheorem{Lemma}[Theorem]{Lemma}

\newtheorem{Claim}{Claim}

\def \leq {\leqslant}
\def \geq {\geqslant}

\let\oldproofname=\proofname
\renewcommand{\proofname}{\rm\bf{\oldproofname}}

\numberwithin{equation}{section}

\allowdisplaybreaks[4]

\begin{document}

\title{A note on non-empty cross-intersecting families}

\author[a]{Menglong Zhang}
\author[a]{Tao Feng}
\affil[a]{School of Mathematics and Statistics, Beijing Jiaotong University, Beijing, 100044, P.R. China}
\affil[ ]{mlzhang@bjtu.edu.cn; tfeng@bjtu.edu.cn}

\date{}
\maketitle

\footnotetext{Supported by NSFC under Grant 12271023}


\begin{abstract}
The families $\mathcal F_1\subseteq \binom{[n]}{k_1},\mathcal F_2\subseteq \binom{[n]}{k_2},\dots,\mathcal F_r\subseteq \binom{[n]}{k_r}$ are said to be cross-intersecting if $|F_i\cap F_j|\geq 1$ for any $1\leq i<j\leq r$ and $F_i\in \mathcal F_i$, $F_j\in\mathcal F_j$. Cross-intersecting families $\mathcal F_1,\mathcal F_2,\dots,\mathcal F_r$ are said to be {\em non-empty} if $\mathcal F_i\neq\emptyset$ for any $1\leq i\leq r$. This paper shows that if $\mathcal F_1\subseteq\binom{[n]}{k_1},\mathcal F_2\subseteq\binom{[n]}{k_2},\dots,\mathcal F_r\subseteq\binom{[n]}{k_r}$ are non-empty cross-intersecting families with $k_1\geq k_2\geq\cdots\geq k_r$ and $n\geq k_1+k_2$, then $\sum_{i=1}^{r}|\mathcal F_i|\leq\max\{\binom{n}{k_1}-\binom{n-k_r}{k_1}+\sum_{i=2}^{r}\binom{n-k_r}{k_i-k_r},\
\sum_{i=1}^{r}\binom{n-1}{k_i-1}\}$. This solves a problem posed by Shi, Frankl and Qian recently. The extremal families attaining the upper bounds are also characterized.
\end{abstract}

\noindent {\bf Keywords}: intersecting family; non-empty cross-intersecting family


\section{Introduction}

Let $n$ and $k$ be integers with $1\leq k\leq n$. Write $[n]=\{1,2,\ldots,n\}$.
Denote by $2^{[n]}$ and $\binom{[n]}{k}$ the power set and the family of all $k$-subsets of $[n]$, respectively. For $1\leq k\leq n$, a family ${\cal F}\subseteq 2^{[n]}$ is said to be {\em $k$-uniform} if every member of $\cal F$ contains exactly $k$ elements, i.e., $\mathcal F\subseteq \binom{[n]}{k}$. Write $\overline{\mathcal F}=\{[n]\setminus F:F\in\mathcal F\}$ for $\mathcal F\subseteq 2^{[n]}$. Two families $\mathcal F,\mathcal F'\subseteq 2^{[n]}$ are said to be {\em isomorphic} if there exists a permutation $\pi$ on $[n]$ such that $\{\{\pi(x):x\in F\}:F\in{\cal F}\}={\cal F}'$.

A family $\mathcal F\subseteq \binom{[n]}{k}$ is said to be {\em intersecting} if $|F_1\cap F_2|\geq 1$ for any $F_1, F_2\in \mathcal F$. The celebrated Erd\H{o}s-Ko-Rado theorem \cite{EKR} determines the size and the structure of the largest intersecting uniform families.

\begin{Theorem}\label{thm:EKR}{\rm \cite{EKR}}
If $n\geq2k$ and $\mathcal F\subseteq \binom{[n]}{k}$ is an intersecting family, then $|\mathcal F|\leq \binom{n-1}{k-1}.$ Moreover, for $n>2k$, the equality holds if and only if $\mathcal F$ is isomorphic to $\{F\in\binom{[n]}{k}:1\in F\}$.
\end{Theorem}

Cross-intersecting families are a variation of intersecting families. Let $r\geq 2$ and $n,k_1,k_2,\dots$, $k_r$ be positive integers. The families $\mathcal F_1\subseteq \binom{[n]}{k_1},\mathcal F_2\subseteq \binom{[n]}{k_2},\dots,\mathcal F_r\subseteq \binom{[n]}{k_r}$ are said to be {\em cross-intersecting} if $|F_i\cap F_j|\geq 1$ for any $1\leq i<j\leq r$ and $F_i\in \mathcal F_i$, $F_j\in\mathcal F_j$. If $\mathcal F$ is intersecting, then the families $\mathcal F_i=\mathcal F$ $(i=1,2\dots,r)$ are cross-intersecting. Another trivial example of cross-intersecting families is $\mathcal F_1\subseteq \binom{[n]}{k_1}$ and $\mathcal F_i=\emptyset$ for every $2\leq i\leq r$. Cross-intersecting families $\mathcal F_1\subseteq\binom{[n]}{k_1},\dots,\mathcal F_r\subseteq\binom{[n]}{k_r}$ are called {\em maximal} if $\mathcal F_1,\dots,(\mathcal F_i\cup\{A\}),\dots,\mathcal F_r$ are not cross-intersecting for any $1\leq i\leq r$ and any $A\in\binom{[n]}{k_i}\setminus\mathcal F_i$.

There are two natural ways to measure the largeness of cross-intersecting families: either by the sum $\sum_{i=1}^{r}|\mathcal F_i|$ or by the product $\prod_{i=1}^{r}|\mathcal F_i|$ (cf. \cite{Brog15,Pyber}) of their sizes. This paper only focuses on the measure of the sum $\sum_{i=1}^{r}|\mathcal F_i|$. Hilton \cite{Hilton77} settled the problem of what cross-intersecting families with the maximum sum are.

\begin{Theorem}\label{thm:cross_inter}{\rm \cite{Hilton77}}
Let $n$ and $k$ be positive integers with $n\geq 2k$. If $\mathcal F_1,\mathcal F_2,\dots,\mathcal F_r\subseteq\binom{[n]}{k}$ are cross-intersecting families, then
$$\sum_{i=1}^{r}|\mathcal F_i|\leq\left\{
\begin{array}{cc}
\binom{n}{k}, & \text{if }r\leq\frac{n}{k};\\
r\binom{n-1}{k-1}, & \text{if }r\geq\frac{n}{k}.
\end{array} \right. $$
If the equality holds, then
\begin{enumerate}
\item[$(1)$] when $r<\frac{n}{k}$, $\mathcal F_1=\binom{[n]}{k}$ and $\mathcal F_2=\dots=\mathcal F_r=\emptyset$;
\item[$(2)$] when $r>\frac{n}{k}$, $\mathcal F_i=\mathcal F$ for every $i\in [r]$, where $\mathcal F\subseteq\binom{[n]}{k}$ is an intersecting family with $|\mathcal F|=\binom{n-1}{k-1}$;
\item[$(3)$] when $r=\frac{n}{k}$, if $r=2$, then $\mathcal F_1=\binom{[n]}{k}\setminus\overline{\mathcal F_2}$ and $\mathcal F_2\subseteq\binom{[n]}{k}$ with $0\leq|\mathcal F_2|\leq\binom{n}{k}$; if $r>2$, then $\mathcal F_1,\mathcal F_2,\dots,\mathcal F_r$ are as in $(1)$ or $(2)$.
\end{enumerate}
\end{Theorem}

Brog \cite{Brog09} gave a short proof for Theorem \ref{thm:cross_inter} and Frankl \cite{Frankl21} recently provided another simple proof by using Katona's circle method. Theorem \ref{thm:cross_inter} shows that one of extremal cross-intersecting families is $\{\binom{[n]}{k},\emptyset,\dots,\emptyset\}$. Cross-intersecting families $\mathcal F_1,\mathcal F_2,\dots,\mathcal F_r$ are said to be {\em non-empty} if $\mathcal F_i\neq\emptyset$ for any $1\leq i\leq r$. It is quite natural to ask what the structure of the largest non-empty cross-intersecting families is. Hilton and Milner gave the following result.

\begin{Theorem}\label{thm:nonempty_2_uniform}{\rm \cite{HM}}
Let $n$ and $k$ be positive integers with $n\geq2k$. If $\mathcal F_1\subseteq\binom{[n]}{k}$ and $\mathcal F_2\subseteq\binom{[n]}{k}$ are non-empty cross-intersecting, then $|\mathcal F_1|+|\mathcal F_2|\leq \binom{n}{k}-\binom{n-k}{k}+1.$
\end{Theorem}

Frankl and Tokushige established the following stronger result by using the Kruskal-Katona Theorem.

\begin{Theorem}\label{thm:nonempty_2_nonuniform}{\rm \cite{FT}}
Let $n,k$ and $l$ be positive integers with $k\geq l$ and $n\geq k+l$. If $\mathcal F_1\subseteq\binom{[n]}{k}$ and $\mathcal F_2\subseteq\binom{[n]}{l}$ are non-empty cross-intersecting, then
\begin{enumerate}
\item[$(1)$] $|\mathcal F_1|+|\mathcal F_2|\leq\binom{n}{k}-\binom{n-l}{k}+1;$
\item[$(2)$] if $|\mathcal F_2|\geq\binom{n-1}{l-1}$, then
$|\mathcal F_1|+|\mathcal F_2|\leq \left\{
\begin{array}{ll}
\binom{n}{k}-\binom{n-k}{k}+1, & \text{if } k=l\geq2;\\
\binom{n-1}{k-1}+\binom{n-1}{l-1}, & \text{otherwise}.
\end{array}
\right.$
\end{enumerate}
\end{Theorem}

Shi, Frankl and Qian \cite{SFQ} gave another generalization of Theorem \ref{thm:nonempty_2_uniform} by extending two families to arbitrary number of families.

\begin{Theorem}\label{thm:nonempty_r_uniform}{\rm \cite[Theorem 1.5]{SFQ}}
Let $r\geq2$ and $n,k$ be positive integers with $n\geq 2k$. If $\mathcal F_1,\mathcal F_2,\dots,\mathcal F_r\subseteq\binom{[n]}{k}$ are non-empty cross-intersecting families, then
$$\sum_{i=1}^{r}|\mathcal F_i|\leq\max\left\{\binom{n}{k}-\binom{n-k}{k}+r-1,\ r\binom{n-1}{k-1}\right\}$$
with equality if and only if
\begin{enumerate}
\item[$(1)$] if $n>2k$, then either there exists $x\in[n]$ such that $\mathcal F_i=\{F\in\binom{[n]}{k}:x\in F\}$ for every $i\in [r]$, or there exist $i^*\in [r]$ and  $S\in\binom{[n]}{k}$ such that $\mathcal F_{i^*}=\{F\in\binom{[n]}{k}:F\cap S\neq\emptyset\}$ and $\mathcal F_i=\{S\}$ for every $i\in[r]\setminus\{i^*\}$;
\item[$(2)$] if $n=2k$, then
\begin{enumerate}
\item[$(i)$] when $r=2$, $\mathcal F_1\subseteq\binom{[n]}{k}$ with $0<|\mathcal F_1|<\binom{n}{k}$ and $\mathcal F_2=\binom{[n]}{k}\setminus\overline{\mathcal F_1}$;
\item[$(ii)$] when $r\geq3$,  $\mathcal F_i=\mathcal F$ for every $i\in [r]$, where $\mathcal F\subseteq\binom{[n]}{k}$ is an intersecting family with $|\mathcal F|=\binom{n-1}{k-1}$.
\end{enumerate}
\end{enumerate}
\end{Theorem}

In this paper, we examine the structure of the largest non-empty cross-intersecting families $\mathcal F_1\subseteq\binom{[n]}{k_1},\mathcal F_2\subseteq\binom{[n]}{k_2},\dots,\mathcal F_r\subseteq\binom{[n]}{k_r}$, where $k_1,k_2,\dots,k_r$ are positive integers. We are to prove the following theorem.

\begin{Theorem}\label{thm:nonempty_k1_k2_kr_other}
Let $r\geq 2$ and $n,k_1,k_2,\dots,k_r$ be positive integers. Let $\mathcal F_1\subseteq\binom{[n]}{k_1},\mathcal F_2\subseteq\binom{[n]}{k_2},\dots,\mathcal F_r\subseteq\binom{[n]}{k_r}$ be non-empty cross-intersecting families with $|\mathcal F_{i^*}|\geq\binom{n-1}{k_{i^*}-1}$ for some $i^*\in[r]$. Let $\bar{k}=\min\{k_i:i\in[r]\setminus\{i^*\}\}$. If $n\geq k_i +k_{i^*}$ for every $i\in[r]\setminus\{i^*\}$, then
$$\sum_{i=1}^{r}|\mathcal F_i|\leq\max\left\{\binom{n}{k_{i^*}}-\binom{n-\bar{k}}{k_{i^*}}+\sum_{i\in[r]\setminus\{{i^*}\}}
\binom{n-\bar{k}}{k_i-\bar{k}},\
\sum_{i=1}^{r}\binom{n-1}{k_i-1}\right\}$$
with equality if and only if
\begin{enumerate}
\item[$(1)$] if $k_i=\bar{k}$ for every $i\in[r]\setminus\{i^*\}$ and $n=\bar{k}+k_{i^*}$, then
\begin{enumerate}
\item[$(i)$] when $r=2$, $\mathcal F_{i^*}=\binom{[n]}{k_{i^*}}\setminus\overline{\mathcal F_{3-i^*}}$ and $1\leq|\mathcal F_{3-i^*}|\leq\binom{n-1}{\bar{k}-1}$;
\item[$(ii)$] when $r>2$ and $n>2\bar{k}$, there exists $x\in[n]$ such that $\mathcal F_{i^*}=\{F\in\binom{[n]}{k_{i^*}}:x\in F\}$ and $\mathcal F_i=\{F\in\binom{[n]}{\bar{k}}:x\in F\}$ for every $i\in[r]\setminus\{i^*\}$;
\item[$(iii)$] when $r>2$ and $n\leq2\bar{k}$, $\mathcal F_i=\mathcal F$ for every $i\in[r]\setminus\{i^*\}$ and $\mathcal F_{i^*}=\binom{[n]}{k_{i^*}}\setminus\overline{\mathcal F}$, where $\mathcal F\subseteq \binom{[n]}{\bar{k}}$ is an intersecting family with $|\mathcal F|=\binom{n-1}{\bar{k}-1}$;
\end{enumerate}
\item[$(2)$] if $n>\bar{k}+k_{i^*}$, then there exists $S\in\binom{[n]}{s}$ with $s=1$ or $\bar{k}$ such that $\mathcal F_{i^*}=\{F\in\binom{[n]}{k_{i^*}}:S\cap F\neq\emptyset\}$ and $\mathcal F_i=\{F\in\binom{[n]}{k_i}:S\subseteq F\}$ for every $i\in[r]\setminus\{i^*\}$.
\end{enumerate}
\end{Theorem}

Applying Theorem \ref{thm:nonempty_k1_k2_kr_other}, we can give a solution to Problems 4.3 in \cite{SFQ} as follows.

\begin{Theorem}\label{thm:nonempty_k1_k2_kr_largest}
Let $r\geq 2$ and $n,k_1,k_2,\dots,k_r$ be positive integers. Let $\mathcal F_1\subseteq\binom{[n]}{k_1},\mathcal F_2\subseteq\binom{[n]}{k_2},\dots,\mathcal F_r\subseteq\binom{[n]}{k_r}$ be non-empty cross-intersecting families with $k_1\geq k_2\geq\cdots\geq k_r$ and $n\geq k_1+k_2$. Then
$$\sum_{i=1}^{r}|\mathcal F_i|\leq\max\left\{\binom{n}{k_1}-\binom{n-k_r}{k_1}+\sum_{i=2}^{r}\binom{n-k_r}{k_i-k_r},\
\sum_{i=1}^{r}\binom{n-1}{k_i-1}\right\}$$
with equality if and only if
\begin{enumerate}
\item[$(1)$] if $n=k_1+k_r$, then
    \begin{enumerate}
        \item[$(i)$] when $r=2$, $\mathcal F_{1}=\binom{[n]}{k_1}\setminus\overline{\mathcal F_{2}}$ and $0<|\mathcal F_{2}|<\binom{n}{k_{2}}$;
        \item[$(ii)$] when $r>2$ and $k_1>k_2=\cdots=k_r$, there exists $x\in[n]$ such that $\mathcal F_i=\{F\in\binom{[n]}{k_i}:x\in F\}$ for every $i\in[r]$;
        \item[$(iii)$] when $r>2$ and $k_1=k_{2}=\dots=k_{r}$, $\mathcal F_i=\mathcal F$ for every $i\in[r]$, where $\mathcal F\subseteq \binom{[n]}{k_1}$ is an intersecting family with $|\mathcal F|=\binom{n-1}{k_1-1}$;
    \end{enumerate}
\item[$(2)$] if $n>k_1+k_r$, then either there exists $x\in[n]$ such that $\mathcal F_{i}=\{F\in\binom{[n]}{k_i}:x\in F\}$ for every $i\in[r]$, or there exists $S\in\binom{[n]}{k_r}$ such that $\mathcal F_{j}=\{F\in\binom{[n]}{k_{j}}:F\cap S\neq\emptyset\}$ for some $j\in[r]$ with $k_j=k_1$ and $\mathcal F_i=\{F\in\binom{[n]}{k_i}:S\subseteq F\}$ for every $i\in[r]\setminus\{j\}$.
\end{enumerate}
\end{Theorem}

\section{Preliminaries}

Let $\prec_L$, or $\prec$ for short, be the lexicographic order on $\binom{[n]}{i}$ where $i\in\{1,2,...,n\}$, that is, for any two sets $A,B\in\binom{[n]}{i}$, $A\prec B$ if and only if $\min\{a: a\in A\setminus B\}<\min\{b: b\in B\setminus A\}$. For a family $\mathcal A\in\binom{[n]}{k}$, let $\mathcal A_L$ denote the family consisting of the first $|\mathcal A|$ $k$-sets in order $\prec$ in $\binom{[n]}{k}$, and call $\mathcal A$ {\em $L$-initial} if $\mathcal A_L =\mathcal A$. A powerful tool in the study of cross-intersecting families is the Kruskal-Katona Theorem (\cite{Katona66,Kruskal}), especially its reformulation due to Hilton \cite{Hilton76}.

\begin{Theorem}\label{thm:Hilton_theorem}{\rm \cite{Hilton76}}
If $\mathcal A\subseteq\binom{[n]}{k}$ and $\mathcal B\subseteq\binom{[n]}{l}$ are cross-intersecting, then $\mathcal A_L$ and $\mathcal B_L$ are cross-intersecting as well.
\end{Theorem}

For positive integers $k,l$ and a set $S\subseteq[n]$, let
$$\mathcal P_{S}^{(l)}=\left\{P\in\binom{[n]}{l}:S\subseteq P\right\} \text{\ and\ }\mathcal R_{S}^{(k)}=\left\{R\in\binom{[n]}{k}:|R\cap S|\geq 1\right\}.$$
Then $|\mathcal P_{S}^{(l)}|=\binom{n-|S|}{l-|S|}$ and $|\mathcal R_{S}^{(k)}|=\binom{n}{k}-\binom{n-|S|}{k}$. Clearly $\mathcal P_{S}^{(l)}$ and $\mathcal R_{S}^{(k)}$ are cross-intersecting. For a positive integer $s$, write $\mathcal P_{s}^{(l)}=\mathcal P_{[s]}^{(l)}$ and $\mathcal R_{s}^{(k)}=\mathcal R_{[s]}^{(k)}$. Clearly $\mathcal P_{s}^{(l)}$ and $\mathcal R_{s}^{(k)}$ are both $L$-initial. Note that $\mathcal P_{s}^{(l)}\subseteq\mathcal P_{s-1}^{(l)}$ for any $2\leq s\leq l$ and $\mathcal R_{s}^{(k)}\subseteq\mathcal R_{s+1}^{(k)}$ for any $s\geq 1$.

The following lemma is a slight generalization of \cite[Lemma 2.1]{SFQ}.

\begin{Lemma}\label{lem:nonempty_2_uniform_1_inter}
Let $n,k$ and $l$ be positive integers with $n\geq k+l$. For any $S\subseteq [n]$ with $1\leq |S|\leq l$, $\mathcal R_{S}^{(k)}$ $($resp. $\mathcal P_{S}^{(l)})$ is the largest family in $\binom{[n]}{k}$ $($resp. $\binom{[n]}{l})$ that is cross-intersecting with $\mathcal P_{S}^{(l)}$ $($resp. $\mathcal R_{S}^{(k)})$.
\end{Lemma}

\begin{proof}
To prove that $\mathcal R_{S}^{(k)}$ is the largest family in $\binom{[n]}{k}$ that is cross-intersecting with $\mathcal P_{S}^{(l)}$, take $A\in\binom{[n]}{k}\setminus\mathcal R_{S}^{(k)}$. It suffices to show that there exists $C\in \mathcal P_{S}^{(l)}$ such that $A\cap C=\emptyset$. Since $A\cap S=\emptyset$, we have $|[n]\setminus(A\cup S)|=n-k-|S|\geq l-|S|$, and so there exists an $(l-|S|)$-set $B\subseteq[n]\setminus(A\cup S)$. Take $C=S\cup B$. Then $C\in\mathcal P_{S}^{(l)}$ and $A\cap C=\emptyset$.

Similarly, to prove that $\mathcal P_{S}^{(l)}$ is the largest family in $\binom{[n]}{l}$ that is cross-intersecting with $\mathcal R_{S}^{(k)}$, take $A\in\binom{[n]}{l}\setminus\mathcal P_{S}^{(l)}$. It suffices to show that there exists $C\in \mathcal R_{S}^{(k)}$ such that $A\cap C=\emptyset$. Write $|A\cap S|=y\leq |S|-1$. If $k\leq |S|-y=|S\setminus A|$, then take $C\in\binom{S\setminus A}{k}$, and so $C\in\mathcal R_{S}^{(k)}$ and $A\cap C=\emptyset$. If $k>|S|-y$, since $|[n]\setminus(A\cup S)|=n-l-|S|+y\geq k-|S|+y>0$, there exists a $(k-|S|+y)$-set $B\subseteq[n]\setminus(A\cup S)$. Take $C=(S\setminus A)\cup B$. Then $C\in\mathcal R_{S}^{(k)}$ and $A\cap C=\emptyset$.
\end{proof}

\begin{Lemma}\label{lem:Comparative_size_3}
Let $r\geq 2$ and $k_1,k_2,\dots,k_r$ be positive integers such that $k_r=\min\{k_i:2\leq i\leq r\}$. Write $$g(s):=\binom{n}{k_1}-\binom{n-s}{k_1}+\sum_{i=2}^{r}\binom{n-s}{k_i-s}.$$
If $n\geq k_1+k_i$ for every $i\in[2,r-1]$ and $n>k_1+k_r$, then $\max\{g(s):1\leq s\leq k_r\}$ is either $g(1)$ or $g(k_r)$.
\end{Lemma}

\begin{proof} When $k_r\in\{1,2\}$, the conclusion is straightforward. Assume that $k_r\geq3$. We claim that there does not exist $s\in[2,k_r-1]$ such that $g(s)\geq g(s+1)$ and $g(s)\geq g(s-1)$. Otherwise, if there exists $s\in[2,k_r-1]$ such that
$$\binom{n}{k_1}-\binom{n-s}{k_1}+\sum_{i=2}^{r}\binom{n-s}{k_i-s}\geq\binom{n}{k_1}-\binom{n-s-1}{k_1}+
\sum_{i=2}^{r}\binom{n-s-1}{k_i-s-1}\ \text{and}$$
$$\binom{n}{k_1}-\binom{n-s}{k_1}+\sum_{i=2}^{r}\binom{n-s}{k_i-s}\geq\binom{n}{k_1}-\binom{n-s+1}{k_1}+
\sum_{i=2}^{r}\binom{n-s+1}{k_i-s+1},$$
which implies
\begin{align}
\label{eqn:1-2}\sum_{i=2}^{r}\binom{n-s-1}{k_i-s}& \geq\binom{n-s-1}{k_1-1}\ \text{and}\\
\label{eqn:1-2-1}\binom{n-s}{k_1-1}& \geq\sum_{i=2}^{r}\binom{n-s}{k_i-s+1}.
\end{align}
$\eqref{eqn:1-2}+\eqref{eqn:1-2-1}\times\frac{n-s-k_1+1}{n-s}$ yields
$$\sum_{i=2}^{r}\binom{n-s-1}{k_i-s}\geq\frac{n-s-k_1+1}{n-s}\sum_{i=2}^{r}\binom{n-s}{k_i-s+1}
=\sum_{i=2}^{r}\frac{n-s-k_1+1}{k_i-s+1}\binom{n-s-1}{k_i-s}.$$
Since $n\geq k_1+k_i$ for every $i\in[2,r-1]$, we have $\frac{n-s-k_1+1}{k_i-s+1}\geq 1$, and since $n>k_1+k_r$, $\frac{n-s-k_1+1}{k_r-s+1}>1$, a contradiction. It is readily checked that the above claim implies that $\max\{g(s):1\leq s\leq k_r\}$ is $g(1)$ or $g(k_r)$.
\end{proof}

\begin{Lemma}\label{lem:binomial}
Let $n$, $k$ and $s$ be positive integers.
\begin{enumerate}
\item[$(1)$] $\binom{n}{k}-\binom{n-s}{k}=\sum\limits_{i=1}^{s}\binom{n-i}{k-1}$.
\item[$(2)$] $\binom{n}{k}-\binom{n-s}{k-s}=\sum\limits_{i=0}^{s-1}\binom{n-i-1}{k-i}$.
\end{enumerate}
\end{Lemma}

\begin{proof}
Apply $\binom{a}{b}=\binom{a-1}{b-1}+\binom{a-1}{b}$ repeatedly. Then we have
$$\binom{n-s}{k}+\sum_{i=1}^{s}\binom{n-i}{k-1}=\binom{n-s+1}{k}+\sum_{i=1}^{s-1}\binom{n-i}{k-1}=\cdots=\binom{n-1}{k}+\binom{n-1}{k-1}=\binom{n}{k}$$
and
\begin{align*}
\binom{n-s}{k-s}+\sum_{i=0}^{s-1}\binom{n-i-1}{k-i}&=\binom{n-s+1}{k-s+1}+\sum_{i=0}^{s-2}\binom{n-i-1}{k-i}
=\cdots=\binom{n}{k}.
\end{align*}
\end{proof}

\begin{Lemma}\label{lem:Comparative_size_4}
Let $r\geq 2$ and $n,k_1,k_2,\dots,k_r$ be positive integers such that $n\geq k_i+k_j$ for any $1\leq i<j\leq r$. Let $\bar{k}=\min\{k_i:i\in[r]\setminus\{2\}\}$.
If $k_1>k_2$ and $\bar{k}>1$, then $$\sum_{i=1}^{r}\binom{n-1}{k_i-1}\geq\binom{n}{k_2}-\binom{n-\bar{k}}{k_2}+\sum_{i\in[r]\setminus\{2\}}
\binom{n-\bar{k}}{k_i-\bar{k}}$$
with equality if and only if $r=2$ and $n=k_1+k_2$.
\end{Lemma}

\begin{proof}
If $r=2$, then $\bar{k}=k_1$ and $\binom{n-1}{k_1-1}+\binom{n-1}{k_2-1}-(\binom{n}{k_2}-\binom{n-k_1}{k_2}+1)=
\binom{n-1}{k_1-1}-\binom{n-1}{k_2}+\binom{n-k_1}{k_2}-1$.
Since $n\geq k_1+k_2$ and $k_1>k_2$, $\binom{n-1}{k_1-1}\geq\binom{n-1}{k_2}$ with equality if and only if $k_1=k_2+1$ or $n=k_1+k_2$. Since $n\geq k_1+k_2$ and $k_2>0$, $\binom{n-k_1}{k_2}-1\geq 0$ with equality if and only if $n=k_1+k_2$. Then the desired conclusion is obtained.

Assume that $r\geq 3$. We have
\begin{align*}
&\sum_{i=1}^{r}\binom{n-1}{k_i-1}-\left(\binom{n}{k_2}-\binom{n-\bar{k}}{k_2}+\binom{n-\bar{k}}{k_1-\bar{k}}+\sum_{i=3}^{r}
\binom{n-\bar{k}}{k_i-\bar{k}}\right)\\
&=\binom{n-1}{k_1-1}-\binom{n-\bar{k}}{k_1-\bar{k}}-\binom{n-1}{k_2}+\binom{n-\bar{k}}{k_2}+\sum_{i=3}^{r}\left(\binom{n-1}{k_i-1}-
\binom{n-\bar{k}}{k_i-\bar{k}}\right)\\
&=\binom{n-1}{k_1-1}-\binom{n-\bar{k}}{k_1-\bar{k}}-\binom{n-1}{k_2}+\binom{n-\bar{k}}{k_2}+\sum_{i=3}^{r}\sum_{j=0}^{\bar{k}-2}
\binom{n-j-2}{k_i-1-j},
\end{align*}
where the last equality follows from Lemma \ref{lem:binomial} $(2)$. If $k_2\geq k_1-\bar{k}$, since $n\geq k_1+k_2$ and $k_1>k_2$, we have $\binom{n-\bar{k}}{k_2}\geq \binom{n-\bar{k}}{k_1-\bar{k}}$ and $\binom{n-1}{k_1-1}\geq\binom{n-1}{k_2}$, which yield the desired inequality. If $k_2<k_1-\bar{k}$, it follows from Lemma \ref{lem:binomial} that $$\binom{n-1}{k_1-1}-\binom{n-\bar{k}}{k_1-\bar{k}}-\binom{n-1}{k_2}+\binom{n-\bar{k}}{k_2}=\sum_{i=1}^{\bar{k}-1}\left(\binom{n-i-1}{k_1-i}
-\binom{n-i-1}{k_2-1}\right).$$
Since $k_2<k_1-\bar{k}$, $k_2-1<k_1-i$ for every $1\leq i\leq \bar{k}-1$, and so $\binom{n-i-1}{k_1-i}\geq\binom{n-i-1}{k_2-1}$. This completes the proof.
\end{proof}

\section{Proof of Theorem \ref{thm:nonempty_k1_k2_kr_other}}

For non-empty cross-intersecting families $\mathcal F_i\subseteq\binom{[n]}{k_i}$, $i\in[r]$, to determine the largest possible value of $\sum_{i=1}^{r}|\mathcal F_i|$, by Theorem \ref{thm:Hilton_theorem}, without loss of generality, one can assume that $\mathcal F_i\subseteq\binom{[n]}{k_i}$, $i\in[r]$, are all $L$-initial.

\begin{Lemma}\label{lem:n>k_j+h_intinal}
Let $r\geq 2$ and $\mathcal F_1\subseteq\binom{[n]}{k_1},\mathcal F_2\subseteq\binom{[n]}{k_2},\dots,\mathcal F_r\subseteq\binom{[n]}{k_r}$ be non-empty $L$-initial cross-intersecting families with $|\mathcal F_{i^*}|\geq\binom{n-1}{k_{i^*}-1}$ for some $i^*\in [r]$. Let $\bar{k}=\min\{k_i:i\in[r]\setminus\{i^*\}\}$. If $n>\bar{k}+k_{i^*}$ and $n\geq k_i+k_{i^*}$ for every $i\in[r]\setminus\{i^*\}$, then
\begin{align}\label{eqn:1}
\sum_{i=1}^{r}|\mathcal F_i|\leq\max\left\{\binom{n}{k_{i^*}}-\binom{n-\bar{k}}{k_{i^*}}+\sum_{i\in[r]\setminus\{i^*\}}
\binom{n-\bar{k}}{k_i-\bar{k}},\
\sum_{i=1}^{r}\binom{n-1}{k_i-1}\right\}
\end{align}
with equality if and only if $\mathcal F_{i^*}=\mathcal R_{s}^{(k_{i^*})}$ and $\mathcal F_i=\mathcal P_{s}^{(k_i)}$ for $i\in[r]\setminus\{i^*\}$, where $s=1$ or $\bar{k}$.
\end{Lemma}

\begin{proof}
Without loss of generality, suppose that $\mathcal F_1,\mathcal F_2,\dots,\mathcal F_r$ are maximal cross-intersecting families.

If $k_{i^*}=1$, write $|\mathcal F_{i^*}|=s\geq1$. Since $\mathcal F_{i^*}$ is $L$-initial, $\mathcal F_{i^*}=\{\{1\},\dots,\{s\}\}=\mathcal R_{s}^{(1)}.$ For any $i\in[r]\setminus\{{i^*}\}$, $\mathcal F_i$ is cross-intersecting with $\mathcal F_{i^*}$, so each element of $\mathcal F_i$ contains $[s]$, i.e., $\mathcal F_i\subseteq \mathcal P_{s}^{(k_i)}$. It follows from the maximality of $\mathcal F_i$ that $\mathcal F_i=\mathcal P_{s}^{(k_i)}$ for every $i\neq i^*$. Note that $s\leq k_i$ for any $i\neq i^*$ and so $s\leq\bar{k}$. Therefore, there is $1\leq s\leq\bar{k}$ such that $\mathcal F_{i^*}=\mathcal R_{s}^{(k_{i^*})}$ and $\mathcal F_i=\mathcal P_{s}^{(k_i)}$ for all $i\in[r]\setminus\{{i^*}\}$.
Apply Lemma \ref{lem:Comparative_size_3} to obtain the upper bound \eqref{eqn:1} for $\sum_{i=1}^{r}|\mathcal F_i|$ and the structure of extremal families.

Assume that $k_{i^*}\geq 2$. Since $\mathcal F_i$ is $L$-initial and nonempty for $i\in[r]\setminus\{{i^*}\}$, we have $[k_i]\in\mathcal F_i$, and so each element of $\mathcal F_{i^*}$ intersects with $[k_i]$ and $\mathcal F_{i^*}\subseteq \mathcal R_{k_i}^{(k_{i^*})}$ for every $i\neq i^*$. Since $\bar{k}=\min\{k_i:i\in[r]\setminus\{{i^*}\}\}$, $\mathcal F_{i^*}\subseteq \mathcal R_{\bar{k}}^{(k_{i^*})}$. By the assumption of $|\mathcal F_{i^*}|\geq\binom{n-1}{k_{i^*}-1}$, $\mathcal R_{1}^{(k_{i^*})}\subseteq\mathcal F_{i^*}$. So $\mathcal R_{1}^{(k_{i^*})}\subseteq\mathcal F_{i^*}\subseteq \mathcal R_{\bar{k}}^{(k_{i^*})}$.

If $\bar{k}=1$, then $\mathcal F_{i^*}=\mathcal R_{1}^{(k_{i^*})}.$
It follows from Lemma \ref{lem:nonempty_2_uniform_1_inter} that $\mathcal F_i\subseteq\mathcal P_1^{(k_i)}$ for $i\neq i^*$. By the maximality of $\mathcal F_i$, $\mathcal F_i=\mathcal P_{1}^{(k_i)}$ for all $i\in[r]\setminus\{i^*\}$, and consequently, $\sum_{i=1}^{r}|\mathcal F_i|=\sum_{i=1}^{r}\binom{n-1}{k_i-1}.$

Assume that $\bar{k}\geq 2$.
Define $s$ to be the largest integer such that $2\leq s+1\leq\bar{k}$ and
$$\mathcal R_{s}^{(k_{i^*})}\subseteq\mathcal F_{i^*}\subseteq\mathcal R_{s+1}^{(k_{i^*})}.$$
$\mathcal R_{s}^{(k_{i^*})}\subseteq\mathcal F_{i^*}$ implies that $\mathcal F_i$ for $i\neq i^*$ is cross-intersecting with $\mathcal R_{s}^{(k_{i^*})}$.
It follows from Lemma \ref{lem:nonempty_2_uniform_1_inter} that $\mathcal F_i\subseteq\mathcal P_{s}^{(k_i)}$ for $i\neq i^*$.
$\mathcal F_{i^*}\subseteq\mathcal R_{s+1}^{(k_{i^*})}$ implies that $\mathcal P_{s+1}^{(k_i)}$ is cross-intersecting with $\mathcal F_{i^*}$ for any $i\neq i^*$.
Since $\mathcal P_{s+1}^{(k_{i_1})}$ is cross-intersecting with $\mathcal F_{i_2}\subseteq\mathcal P_{s}^{(k_{i_2})}$ for any $i_1\neq i_2$ and $i_1,i_2\in[r]\setminus\{{i^*}\}$, $\mathcal P_{s+1}^{(k_i)}\subseteq\mathcal F_i$ for $i\in[r]\setminus\{{i^*}\}$ by the maximality of $\mathcal F_i$.
Therefore, for any $i\in[r]\setminus\{{i^*}\}$, $$\mathcal P_{s+1}^{(k_i)}\subseteq\mathcal F_i\subseteq\mathcal P_{s}^{(k_i)}.$$
Let
$\mathcal G_{i^*}=\mathcal F_{i^*}\setminus\mathcal R_{s}^{(k_{i^*})}$ and $\mathcal G_i=\mathcal F_i\setminus\mathcal P_{s+1}^{(k_i)}$ for  $i\in[r]\setminus\{{i^*}\}$.
Then
$$\sum_{i=1}^{r}|\mathcal F_i|=|\mathcal R_{s}^{(k_{i^*})}|+\sum_{i\in[r]\setminus\{{i^*}\}}|\mathcal P_{s+1}^{(k_i)}|+\sum_{i=1}^{r}|\mathcal G_i|.$$
Our goal is to maximize $\sum_{i=1}^{r}|\mathcal G_i|$.

Since $\mathcal G_{i^*}\subseteq\mathcal R_{s+1}^{(k_{i^*})}\setminus\mathcal R_{s}^{(k_{i^*})}=\{R\in\binom{[n]}{k_{i^*}}:R\cap[s+1]=\{s+1\}\}=\{T\cup\{s+1\}:T\in \binom{[s+2,n]}{k_{i^*}-1}\}$ and for $i\in[r]\setminus\{{i^*}\}$, $\mathcal G_i\subseteq\mathcal P_{s}^{(k_i)}\setminus\mathcal P_{s+1}^{(k_i)}=\{P\in\binom{[n]}{k_i}:P\cap[s+1]=[s]\}=\{T\cup[s]:T\in\binom{[s+2,n]}{k_i-s}\}$, we construct an auxiliary $r$-partite graph $G=(\mathcal X_1,\mathcal X_2,\dots,\mathcal X_r,E(G))$ where $\mathcal X_{i^*}=\binom{[s+2,n]}{k_{i^*}-1}$ and $\mathcal X_i=\binom{[s+2,n]}{k_i-s}$ for $i\in[r]\setminus\{{i^*}\}$ such that for $X_{i^*}\in \mathcal X_{i^*}$ and $X_i\in \mathcal X_i$ with $i\neq i^*$, $X_{i^*}X_i$ is an edge if and only if $X_{i^*}\cap X_i=\emptyset$, and there is no edge between $\mathcal X_{i_1}$ and $\mathcal X_{i_2}$ for $i_1,i_2\in[r]\setminus\{{i^*}\}$.
Clearly $G$ can be also regarded as a bipartite graph with parts $\mathcal X_{i^*}$ and $V(G)\setminus\mathcal X_{i^*}$. Note that the vertex set of $G$ is a slight abuse of notation, since it is possible that $\mathcal X_{i_1}=\mathcal X_{i_2}$ for some $1\leq i_1<i_2\leq r$.
Let $\mathcal I_{i^*}=\{R\setminus\{s+1\}:R\in\mathcal G_{i^*}\}$ and $\mathcal I_i=\{P\setminus[s]:P\in\mathcal G_i\}$ for $i\in[r]\setminus\{{i^*}\}$. Then $\mathcal I_i\subseteq \mathcal X_i$ for $i\in[r]$ and $\mathcal I=\bigcup_{i=1}^{r}\mathcal I_i$ is an independent set of $G$.
On the other hand, if $\mathcal I'$ is an independent set of $G$ and $\mathcal I'_i=\mathcal I'\cap \mathcal X_i$ for $i\in[r]$, then $\mathcal G'_1,\dots,\mathcal G'_r$ are cross-intersecting, where $\mathcal G'_{i^*}=\{\{s+1\}\cup X_{i^*}:X_{i^*}\in\mathcal I'_{i^*}\}$ and $\mathcal G'_i=\{[s]\cup X_i:X_i\in\mathcal I'_i\}$ for $i\in[r]\setminus\{{i^*}\}$. Therefore, to maximize $\sum_{i=1}^{r}|\mathcal G_i|$, it suffices to examine the largest independent set of $G$.

\begin{Claim}\label{clm:bipartite_neighbour}
For any $\mathcal Q\subseteq \mathcal X_{i^*}$ and any $i\in[r]\setminus\{{i^*}\}$, $|N_{\mathcal X_i}(\mathcal Q)|\geq\frac{|\mathcal X_i||\mathcal Q|}{|\mathcal X_{i^*}|}$, where $N_{\mathcal X_i}(\mathcal Q)=\{X_i\in \mathcal X_i: \text{there is } X_{i^*}\in \mathcal Q\text{ such that }X_{i^*}X_i\text{ is an edge}\}$.
Moreover, if $n>k_i+k_{i^*}$, then the equality holds if and only if $\mathcal Q=\emptyset$ or $\mathcal Q=\mathcal X_{i^*}$.
\end{Claim}

\begin{proof}[{\bf Proof of Claim \ref{clm:bipartite_neighbour}}]
By the definition of the graph $G$, for any $i\in[r]\setminus\{{i^*}\}$, the vertices in the same partite set have the same degree in the induced bipartite subgraph $G[\mathcal X_{i^*},\mathcal X_i]$.
Let $d_{i^*}$ and $d_i$ denote the degree of the vertex in $\mathcal X_{i^*}$ and $\mathcal X_i$, respectively.
So $e(G[\mathcal X_{i^*},\mathcal X_i])=d_{i^*}|\mathcal X_{i^*}|=d_i|\mathcal X_i|$, where $e(G[\mathcal X_{i^*},\mathcal X_i])$ is the number of edges in $G[\mathcal X_{i^*},\mathcal X_i]$.
On the other hand, $d_{i^*}|\mathcal Q|\leq d_i|N_{\mathcal X_i}(\mathcal Q)|$. So $|N_{\mathcal X_i}(\mathcal Q)|\geq\frac{|\mathcal X_i||\mathcal Q|}{|\mathcal X_{i^*}|}$.
When $n>k_i+k_{i^*}$, the induced bipartite subgraph $G[\mathcal X_{i^*},\mathcal X_i]$ is connected.
Hence, the equality holds if and only if $\mathcal Q=\emptyset$ or $\mathcal Q=\mathcal X_{i^*}$.
\end{proof}

Since $\mathcal I_i\cap N_{\mathcal X_i}(\mathcal I_{i^*})=\emptyset$ for $i\neq i^*$, applying Claim \ref{clm:bipartite_neighbour}, we have
$$\frac{|\mathcal I_{i^*}|\binom{n-s-1}{k_i-s}}{\binom{n-s-1}{k_{i^*}-1}}+|\mathcal I_i|\leq|N_{\mathcal X_i}(\mathcal I_{i^*})|+|\mathcal I_i|\leq|\mathcal X_i|=\binom{n-s-1}{k_i-s}.$$
Hence, for $i\neq i^*$,
$$|\mathcal I_i|\leq\binom{n-s-1}{k_i-s}\left(1-\frac{|\mathcal I_{i^*}|}{\binom{n-s-1}{k_{i^*}-1}}\right),$$
and when $n>k_i +k_{i^*}$, the equality holds if and only if $\mathcal I_{i^*}=\emptyset$ or $\mathcal X_{i^*}$.

Let $H_1=\{i\in[r]:n>k_i+k_{i^*},i\neq {i^*}\}$ and $H_2=\{i\in[r]:n=k_i+k_{i^*},i\neq {i^*}\}$. Note that $H_1\neq\emptyset$ because of $n>\bar{k}+k_{i^*}$.

\noindent{\bf \underline{Case 1 $H_2=\emptyset$.}}
We claim that $\mathcal I_{i^*}$ is either $\emptyset$ or $\mathcal X_{i^*}$ when $\mathcal I$ is the largest independent set of $G$.
If $\mathcal I_{i^*}=\emptyset$, then $\mathcal I_i=\mathcal X_i$ for $i\in[r]\setminus\{{i^*}\}$ and so $$|\mathcal I|=\sum_{i\in[r]\setminus\{{i^*}\}}\binom{n-s-1}{k_i-s}.$$
If $\mathcal I_{i^*}=\mathcal X_{i^*}$, then $\mathcal I_i=\emptyset$ for $i\in[r]\setminus\{{i^*}\}$ and so $$|\mathcal I|=\binom{n-s-1}{k_{i^*}-1}.$$
If $\mathcal I_{i^*}\neq\emptyset$ and $\mathcal I_{i^*}\neq \mathcal X_{i^*}$, then
\begin{align*}
|\mathcal I|=\sum_{i=1}^{r}|\mathcal I_i|&<|\mathcal I_{i^*}|+\sum_{i\in[r]\setminus\{{i^*}\}}\binom{n-s-1}{k_i-s}\left(1-\frac{|\mathcal I_{i^*}|}{\binom{n-s-1}{k_{i^*}-1}}\right).
\end{align*}
If $\sum_{i\in[r]\setminus\{{i^*}\}}\binom{n-s-1}{k_i-s}\leq\binom{n-s-1}{k_{i^*}-1}$, then $|\mathcal I|<\binom{n-s-1}{k_{i^*}-1}$. If $\sum_{i\in[r]\setminus\{{i^*}\}}\binom{n-s-1}{k_i-s}>\binom{n-s-1}{k_{i^*}-1}$, then
\begin{align*}
|\mathcal I|<\sum_{i\in[r]\setminus\{{i^*}\}}\binom{n-s-1}{k_i-s}+|\mathcal I_{i^*}|\left(1-\frac{\sum_{i\in[r]\setminus\{{i^*}\}}\binom{n-s-1}{k_i-s}}{\binom{n-s-1}{k_{i^*}-1}}\right)
<\sum_{i\in[r]\setminus\{{i^*}\}}\binom{n-s-1}{k_i-s}.
\end{align*}
To sum up, if $\mathcal I_{i^*}\neq\emptyset$ and $\mathcal I_{i^*}\neq \mathcal X_{i^*}$, then $$|\mathcal I|<\max\left\{\binom{n-s-1}{k_{i^*}-1},\sum_{i\in[r]\setminus\{{i^*}\}}\binom{n-s-1}{k_i-s}\right\}.$$ Therefore, if $\mathcal I$ is the largest independent set of $G$, then $\mathcal I_{i^*}=\emptyset$ or $\mathcal X_{i^*}$, and so $\mathcal I=V(G)\setminus\mathcal X_{i^*}$ or $\mathcal X_{i^*}$, which yields that either $\mathcal G_{i^*}=\emptyset$ and $\mathcal G_i=\{[s]\cup X_i:X_i\in\binom{[s+2,n]}{k_i-s}\}$ for all $i\neq i^*$, or $\mathcal G_{i^*}=\{\{s+1\}\cup X_{i^*}:X_{i^*}\in\binom{[s+2,n]}{k_{i^*}-1}\}$ and $\mathcal G_i=\emptyset$ for all $i\neq i^*$.
That is, either $\mathcal F_{i^*}=\mathcal R_{s}^{(k_{i^*})}$ and $\mathcal F_i=\mathcal P_{s}^{(k_i)}$ for all $i\in[r]\setminus\{{i^*}\}$, or $\mathcal F_{i^*}=\mathcal R_{s+1}^{(k_{i^*})}$ and $\mathcal F_i=\mathcal P_{s+1}^{(k_i)}$ for all $i\in[r]\setminus\{{i^*}\}$.
Hence, there is $1\leq s\leq\bar{k}$ such that $\mathcal F_{i^*}=\mathcal R_{s}^{(k_{i^*})}$ and $\mathcal F_i=\mathcal P_{s}^{(k_i)}$ for all $i\in[r]\setminus\{{i^*}\}$.
By Lemma \ref{lem:Comparative_size_3}, we obtain the upper bound \eqref{eqn:1} for $\sum_{i=1}^{r}|\mathcal F_i|$ and the structure of extremal families.

\noindent{\bf \underline{Case 2 $H_2\neq\emptyset$.}} For any $i\in H_2$, the induced bipartite subgraph $G[\mathcal X_i,\mathcal X_{i^*}]$ is a perfect matching, and so $|\mathcal I_i|+|\mathcal I_{i^*}|=|\mathcal I_i|+|N_{\mathcal X_i}(\mathcal I_{i^*})|\leq|\mathcal X_i|=\binom{n-s-1}{k_i-s}$. Thus
\begin{align*}
|\mathcal I|=\sum_{i=1}^{r}|\mathcal I_i|&\leq \sum_{i\in H_1}\binom{n-s-1}{k_i-s}\left(1-\frac{|\mathcal I_{i^*}|}{\binom{n-s-1}{k_{i^*}-1}}\right)+\sum_{i\in H_2}|\mathcal I_i|+|\mathcal I_{i^*}|\\
&\leq\sum_{i\in H_1}\binom{n-s-1}{k_i-s}+\sum_{i\in H_2}\binom{n-s-1}{k_i-s}=\sum_{i\in[r]\setminus\{{i^*}\}}\binom{n-s-1}{k_i-s},
\end{align*}
with equality if and only if $\mathcal I_{i^*}=\emptyset$ and $\mathcal I_i=\mathcal X_i$ for $i\in[r]\setminus\{{i^*}\}$. Then $\mathcal G_{i^*}=\emptyset$ and $\mathcal G_i=\{[s]\cup X_i:X_i\in\binom{[s+2,n]}{k_i-s}\}$ for $i\in[r]\setminus\{{i^*}\}$, and hence $\mathcal F_{i^*}=\mathcal R_{s}^{(k_{i^*})}$ and $\mathcal F_i=\mathcal P_{s}^{(k_i)}$ for $i\in[r]\setminus\{{i^*}\}$.
Therefore, there is $1\leq s\leq\bar{k}$ such that $\mathcal F_{i^*}=\mathcal R_{s}^{(k_{i^*})}$ and $\mathcal F_i=\mathcal P_{s}^{(k_i)}$ for all $i\in[r]\setminus\{{i^*}\}$.
By Lemma \ref{lem:Comparative_size_3}, we obtain the upper bound \eqref{eqn:1} for $\sum_{i=1}^{r}|\mathcal F_i|$ and the structure of extremal families.
\end{proof}

Lemma \ref{lem:n>k_j+h_intinal} gives the structure of the largest non-empty $L$-initial cross-intersecting families. We will generalize Lemma \ref{lem:n>k_j+h_intinal} to general non-empty cross-intersecting families. To this end, we need a notation. For $\mathcal A\subseteq\binom{[n]}{k}$, let
$$\mathcal D_{l}(\mathcal A)=\{D\in\binom{[n]}{l}:\text{there exists }A\in\mathcal A\text{ such that }A\cap D=\emptyset\}.$$
Clearly $\mathcal A\subseteq\binom{[n]}{k}$ and $\mathcal B\subseteq\binom{[n]}{l}$ are cross-intersecting if and only if $\mathcal A\subseteq\binom{[n]}{k}\setminus\mathcal D_k(\mathcal B)$ or $\mathcal B\subseteq\binom{[n]}{l}\setminus\mathcal D_l(\mathcal A)$. The following lemma was implicitly given in \cite{FG,Mors} and was stated explicitly in \cite[Proposition 2.3]{SFQ}.

\begin{Lemma}\label{lem:non_intersecting}{\rm \cite{FG,Mors,SFQ}}
If $n>k+l$ and $\mathcal A\subseteq\binom{[n]}{k}$ with $|\mathcal A|=\binom{n-s}{k-s}$ for some $1\leq s\leq k$, then $|\mathcal D_{l}(\mathcal A)|\geq\binom{n-s}{l}$ with equality if and only if $\mathcal A=\mathcal P_{S}^{(k)}$ for some $S\in\binom{[n]}{s}$.
\end{Lemma}

\begin{Lemma}\label{lem:n>k_j+h}
Let $r\geq 2$ and $\mathcal F_1\subseteq\binom{[n]}{k_1},\mathcal F_2\subseteq\binom{[n]}{k_2},\dots,\mathcal F_r\subseteq\binom{[n]}{k_r}$ be non-empty cross-intersecting families with $|\mathcal F_{i^*}|\geq\binom{n-1}{k_{i^*}-1}$ for some ${i^*}\in [r]$. Let $\bar{k}=\min\{k_i:i\in[r]\setminus\{{i^*}\}\}$. If $n>\bar{k}+k_{i^*}$ and $n\geq k_i+k_{i^*}$ for every $i\in[r]\setminus\{{i^*}\}$, then
\begin{align}\label{eqn:2}
\sum_{i=1}^{r}|\mathcal F_i|\leq\max\left\{\binom{n}{k_{i^*}}-\binom{n-\bar{k}}{k_{i^*}}+\sum_{i\in[r]\setminus\{{i^*}\}}
\binom{n-\bar{k}}{k_i-\bar{k}},\
\sum_{i=1}^{r}\binom{n-1}{k_i-1}\right\}
\end{align}
with equality if and only if there is some $S\in\binom{[n]}{s}$ with $s=1$ or $\bar{k}$ such that $\mathcal F_{i^*}=\mathcal R_{S}^{(k_{i^*})}$ and $\mathcal F_i=\mathcal P_{S}^{(k_{i})}$ for  $i\in[r]\setminus\{{i^*}\}$.
\end{Lemma}

\begin{proof}
The upper bound \eqref{eqn:2} for $\sum_{i=1}^{r}|\mathcal F_i|$ follows from Theorem \ref{thm:Hilton_theorem} and Lemma \ref{lem:n>k_j+h_intinal}. Assume that the equality in \eqref{eqn:2} holds. Then by Theorem \ref{thm:Hilton_theorem} and Lemma \ref{lem:n>k_j+h_intinal}, $|\mathcal F_{i^*}|=\binom{n}{k_{i^*}}-\binom{n-s}{k_{i^*}}$ and $|\mathcal F_i|=\binom{n-s}{k_i-s}$ for $i\in[r]\setminus\{{i^*}\}$, where $s=1$ or $\bar{k}$.
Let $H_1=\{i\in[r]:n>k_i+k_{i^*},i\neq {i^*}\}$ and $H_2=\{i\in[r]:n=k_i+k_{i^*},i\neq {i^*}\}$. Since $n>\bar{k}+k_{i^*}$, $H_1\neq\emptyset$.
For any $i\in H_1$, $\binom{n}{k_{i^*}}-|\mathcal D_{k_{i^*}}(\mathcal F_i)|\geq|\mathcal F_{i^*}|$, i.e., $|\mathcal D_{k_{i^*}}(\mathcal F_i)|\leq\binom{n-s}{k_{i^*}}$.
By Lemma \ref{lem:non_intersecting}, $|\mathcal D_{k_{i^*}}(\mathcal F_i)|=\binom{n-s}{k_{i^*}}$ and  there exists $S_i\in\binom{[n]}{s}$ for $i\in H_1$ such that $\mathcal F_i=\mathcal P_{S_i}^{(k_i)}$.

We claim that $S_{i_1}=S_{i_2}$, written as $S$, for any $i_1,i_2\in H_1$, and $\mathcal F_{i^*}=\mathcal R_{S}^{(k_{i^*})}$. Indeed, given any $h\in H_1$, let $S=S_h$. Since $\mathcal F_{i^*}$ and $\mathcal F_h=\mathcal P_{S_h}^{(k_h)}$ are cross-intersecting and $|\mathcal F_{i^*}|=\binom{n}{k_{i^*}}-\binom{n-s}{k_{i^*}}$, it follows from Lemma \ref{lem:nonempty_2_uniform_1_inter} that $\mathcal F_{i^*}=\mathcal R_{S}^{(k_{i^*})}$. For any $i\in H_1\setminus\{h\}$, since $\mathcal F_{i^*}$ and $\mathcal F_i$ are cross-intersecting and $|\mathcal F_i|=\binom{n-s}{k_i-s}$, by Lemma \ref{lem:nonempty_2_uniform_1_inter}, we have $\mathcal F_i=\mathcal P_{S}^{(k_i)}$.

For every $i\in H_2$, since $\mathcal F_{i^*}$ and $\mathcal F_i$ are cross-intersecting, $\mathcal F_i\subseteq\binom{[n]}{k_i}\setminus\overline{\mathcal F_{i^*}}$. Since $|\mathcal F_i|=\binom{n-s}{k_i-s}=|\binom{[n]}{k_i}\setminus\overline{\mathcal F_{i^*}}|$, we have $\mathcal F_i=\binom{[n]}{k_i}\setminus\overline{\mathcal F_{i^*}}=\mathcal P_{S}^{(k_i)}$.
\end{proof}

To complete the proof of Theorem \ref{thm:nonempty_k1_k2_kr_other}, we still need to examine the case of $n=\bar{k}+k_{i^*}$ and $n\geq k_i+k_{i^*}$ for every $i\in[r]\setminus\{{i^*}\}$, where $\bar{k}=\min\{k_i:i\in[r]\setminus\{{i^*}\}\}$. In this case $k_i=\bar{k}$ for every $i\in[r]\setminus\{i^*\}$. The following theorem will be used.

\begin{Theorem}\label{thm:cross_inter_famliysize}{\rm \cite[Theorem 1.4]{SFQ}}
Let $n,k,l$ and $\tau$ be positive integers with $n\geq k+l$ and $l\geq\tau$. Let $c$ be a positive integer. If $\mathcal A\subseteq\binom{[n]}{k}$ and $\mathcal B\subseteq\binom{[n]}{l}$ are cross-intersecting and $\binom{n-\tau}{l-\tau}\leq|\mathcal B|\leq\binom{n-1}{l-1}$, then
$$|\mathcal A|+c|\mathcal B|\leq\max\left\{\binom{n}{k}-\binom{n-\tau}{k}+c\binom{n-\tau}{l-\tau},
\binom{n-1}{k-1}+c\binom{n-1}{l-1}\right\}$$
with equality if and only if, up to isomorphism, one of the following holds:
\begin{enumerate}
\item[$(1)$] when $n>k+l$, $\mathcal A=\mathcal R_{s}^{(k)}$ and $\mathcal B=\mathcal P_{s}^{(l)}$, where
    $$s=\left\{
    \begin{array}{ll}
    1, & {\text if\ } \binom{n}{k}-\binom{n-r}{k}+c\binom{n-\tau}{l-\tau}<\binom{n-1}{k-1}+c\binom{n-1}{l-1};  \\
    \tau, & {\text if\ } \binom{n}{k}-\binom{n-r}{k}+c\binom{n-\tau}{l-\tau}>\binom{n-1}{k-1}+c\binom{n-1}{l-1};\\
    1 \text{\ or\ } \tau, & {\text if\ } \binom{n}{k}-\binom{n-r}{k}+c\binom{n-\tau}{l-\tau}= \binom{n-1}{k-1}+c\binom{n-1}{l-1};
    \end{array}
    \right.$$
\item[$(2)$] when $n=k+l$, $\mathcal B\subseteq\binom{[n]}{l}$ with $|\mathcal B|=\binom{n-\tau}{l-\tau}$ if $c<1$ or $\binom{n-\tau}{l-\tau}\leq|\mathcal B|\leq\binom{n-1}{l-1}$ if $c=1$ or $|\mathcal B|=\binom{n-1}{l-1}$ if $c>1$, and $\mathcal A=\binom{[n]}{k}\setminus\overline{\mathcal B}$.
\end{enumerate}
\end{Theorem}

\begin{Lemma}\label{lem:n=k_j+h}
Let $r\geq 2$ and $\mathcal F_1\subseteq\binom{[n]}{k_1},\mathcal F_{2}\subseteq\binom{[n]}{k_{2}},\dots,\mathcal F_r\subseteq\binom{[n]}{k_r}$ be non-empty cross-intersecting families with $|\mathcal F_{i^*}|\geq\binom{n-1}{k_{i^*}-1}$ for some $i^*\in[r]$ and $k_i=h$ for every $i\in [r]\setminus [i^*]$ where $h$ is a positive integer. If $n=h+k_{i^*}$,
then
$$\sum_{i=1}^{r}|\mathcal F_i|\leq\binom{n-1}{k_{i^*}-1}+(r-1)\binom{n-1}{h-1}$$
with equality if and only if
\begin{enumerate}
\item[$(1)$] when $r=2$, $\mathcal F_{{i^*}}=\binom{[n]}{k_{i^*}}\setminus\overline{\mathcal F_{3-{i^*}}}$ and $1\leq|\mathcal F_{3-{i^*}}|\leq\binom{n-1}{h-1}$;
\item[$(2)$] when $r>2$ and $n>2h$, there exists $x\in[n]$ such that $\mathcal F_{i^*}=\{F\in\binom{[n]}{k_{i^*}}:x\in F\}$ and $\mathcal F_i=\{F\in\binom{[n]}{h}:x\in F\}$ for every $i\in[r]\setminus\{{i^*}\}$;
\item[$(3)$] when $r>2$ and $n\leq2h$, $\mathcal F_i=\mathcal F$ for every $i\in[r]\setminus\{{i^*}\}$ and $\mathcal F_{i^*}=\binom{[n]}{k_{i^*}}\setminus\overline{\mathcal F}$, where $\mathcal F\subseteq\binom{[n]}{h}$ is an intersecting family with $|\mathcal F|=\binom{n-1}{h-1}$.
\end{enumerate}
\end{Lemma}

\begin{proof}
Since $|\mathcal F_{i^*}|\geq\binom{n-1}{k_{i^*}-1}$, by Theorem \ref{thm:Hilton_theorem} and Lemma \ref{lem:nonempty_2_uniform_1_inter}, $|\mathcal F_i|\leq\binom{n-1}{k_i-1}=\binom{n-1}{h-1}$ for every $i\in[r]\setminus\{{i^*}\}$.
If $r=2$, then apply Theorem \ref{thm:cross_inter_famliysize} with $k=k_{i^*}$, $l=\tau=h$ and $c=1$ to complete the proof. Assume that $r>2$. Since $n=h+k_{i^*}$,
\begin{align}\label{eqn:12-29}
\binom{n-1}{k_{i^*}-1}+(r-1)\binom{n-1}{h-1}&=\binom{n-1}{k_{i^*}-1}+\binom{n-1}{k_{i^*}}+(r-2)
\binom{n-1}{h-1}
\geq\binom{n}{k_{i^*}}+r-2.
\end{align}
It follows from Theorem \ref{thm:cross_inter_famliysize} with $k=k_{i^*}$, $l=\tau=h$ and $c=r-1>1$ that for any $i\in [r]\setminus\{i^*\}$,
\begin{align*}
|\mathcal F_{i^*}|+(r-1)|\mathcal F_i|&\leq \max\left\{\binom{n}{k_{i^*}}+r-2,\ \binom{n-1}{k_{i^*}-1}+(r-1)\binom{n-1}{h-1}\right\}\\
&=\binom{n-1}{k_{i^*}-1}+(r-1)\binom{n-1}{h-1}
\end{align*}
with equality if and only if $|\mathcal F_i|=\binom{n-1}{h-1}$ and $\mathcal F_{i^*}=\binom{[n]}{k_{i^*}}\setminus\overline{\mathcal F_i}$. Therefore,
$$\sum_{i=1}^{r}|\mathcal F_i|=\frac{1}{r-1}\sum_{i\in[r]\setminus\{{i^*}\}}(|\mathcal F_{i^*}|+(r-1)|\mathcal F_i|)\leq \binom{n-1}{k_{i^*}-1}+(r-1)\binom{n-1}{h-1}$$
with equality if and only if $\mathcal F_i=\mathcal F$ with $|\mathcal F|=\binom{n-1}{h-1}$ for every $i\in[r]\setminus\{{i^*}\}$ and $\mathcal F_{i^*}=\binom{[n]}{k_{i^*}}\setminus\overline{\mathcal F}$, where $\mathcal F\subseteq\binom{[n]}{h}$ is an intersecting family because of the cross-intersecting property of $\mathcal F_i$ for $i\in[r]\setminus\{{i^*}\}$. When $n>2h$, by Theorem \ref{thm:EKR}, $\mathcal F=\{F\in\binom{[n]}{h}:x\in F\}$ for some $x\in[n]$, and so $\mathcal F_{i^*}=\binom{[n]}{k_{i^*}}\setminus\overline{\mathcal F}=\{F\in\binom{[n]}{k_{i^*}}:x\in F\}$.
\end{proof}

Now we are in a position to give a proof of Theorem \ref{thm:nonempty_k1_k2_kr_other}.

\begin{proof}[{\bf Proof of Theorem \ref{thm:nonempty_k1_k2_kr_other}}]
When $k_i=\bar{k}$ for every $i\in[r]\setminus\{i^*\}$ and $n=\bar{k}+k_{i^*}$, it follows from \eqref{eqn:12-29} that $\binom{n-1}{k_{i^*}-1}+(r-1)\binom{n-1}{\bar{k}-1}\geq\binom{n}{k_{i^*}}+r-2$. Combine Lemmas \ref{lem:n>k_j+h} and \ref{lem:n=k_j+h} to complete the proof.
\end{proof}

\section{Proof of Theorem \ref{thm:nonempty_k1_k2_kr_largest}}

We are ready to prove Theorem \ref{thm:nonempty_k1_k2_kr_largest} by employing Theorem \ref{thm:nonempty_k1_k2_kr_other}.

\begin{proof}[{\bf Proof of Theorem \ref{thm:nonempty_k1_k2_kr_largest}}]
If $|\mathcal F_i|<\binom{n-1}{k_i-1}$ for every $i\in[r]$, then $\sum_{i=1}^{r}|\mathcal F_i|<\sum_{i=1}^{r}\binom{n-1}{k_i-1}\leq\max\{\binom{n}{k_1}-\binom{n-k_r}{k_1}+\sum_{i=2}^{r}\binom{n-k_r}{k_i-k_r},
\sum_{i=1}^{r}\binom{n-1}{k_i-1}\}$.

Assume that there exists ${i^*}\in[r]$ such that $|\mathcal F_{i^*}|\geq\binom{n-1}{k_{i^*}-1}$. We shall compare the sizes of extremal families coming from Theorems \ref{thm:nonempty_k1_k2_kr_other}.

\noindent\underline{{\bf Case 1 $r=2$} and $n=k_1+k_2$.} In this case, $\binom{n}{k_1}=\binom{n}{k_2}=\binom{n-1}{k_1-1}+\binom{n-1}{k_2-1}$.
It follows from $(i)$ of Theorem \ref{thm:nonempty_k1_k2_kr_other} $(1)$ that
$|\mathcal F_1|+|\mathcal F_2|\leq\binom{n}{k_1}$ with equality if and only if $$\mathcal F_{1}=\binom{[n]}{k_1}\setminus\overline{\mathcal F_{2}}\text{ with } 1\leq|\mathcal F_{2}|\leq\binom{n-1}{k_{2}-1}$$ or $$\mathcal F_{2}=\binom{[n]}{k_2}\setminus\overline{\mathcal F_{1}} \text{ with }1\leq|\mathcal F_{1}|\leq\binom{n-1}{k_{1}-1}.$$
The later case is equivalent to $\mathcal F_{1}=\binom{[n]}{k_1}\setminus\overline{\mathcal F_{2}}$ with $\binom{n-1}{k_{2}-1}\leq|\mathcal F_{2}|\leq\binom{n}{k_2}-1$.
Therefore, the equality holds if and only if
$$\mathcal F_{1}=\binom{[n]}{k_1}\setminus\overline{\mathcal F_{2}}\text{ with }0<|\mathcal F_{2}|<\binom{n}{k_2}.$$

\noindent\underline{{\bf Case 2 $r>2$} and $n=k_1+k_r$.} In this case, $\binom{n-1}{k_{1}-1}+(r-1)\binom{n-1}{k_r-1}\geq\binom{n}{k_{1}}+r-2$.
Since $n\geq k_1+k_2$ and $k_2\geq\cdots\geq k_r$, we have $k_2=\cdots=k_r$.

If $k_1=k_2$, then by Theorem \ref{thm:nonempty_r_uniform} $(2)$$(ii)$, $\sum_{i=1}^{r}|\mathcal F_i|\leq\sum_{i=1}^{r}\binom{n-1}{k_i-1}$ with equality if and only if $\mathcal F_i=\mathcal F$ for $i\in[r]$, where $\mathcal F$ is an intersecting family with $|\mathcal F|=\binom{n-1}{k_1-1}$.

Assume that $k_1>k_2$. By Theorem \ref{thm:nonempty_k1_k2_kr_other} $(1)$$(ii)$ and $(2)$, if $\sum_{i=1}^{r}|\mathcal F_i|$ reaches the largest possible value, then up to isomorphism, either $\mathcal F_i=\mathcal P_1^{(k_i)}$ for every $i\in[r]$, or $\mathcal F_{i^*}=\mathcal R_{k_r}^{(k_{i^*})}$ with $i^*\in[2,r]$ and $\mathcal F_i=\mathcal P_{k_r}^{(k_i)}$ for every $i\in[r]\setminus\{i^*\}$.
If $k_r=1$, then the above two classes of extremal families are the same and so $\sum_{i=1}^{r}|\mathcal F_i|\leq\sum_{i=1}^{r}\binom{n-1}{k_i-1}$ with equality if and only if $\mathcal F_i=\mathcal P_1^{(k_i)}$ for every $i\in[r]$ up to isomorphism.
If $k_r>1$, then
\begin{align*}
\sum_{i=1}^{r}|\mathcal P_1^{(k_i)}|&=\sum_{i=1}^{r}\binom{n-1}{k_i-1}=\binom{n-1}{k_1-1}+\binom{n-1}{k_2-1}+\sum_{i=3}^{r}\binom{n-1}{k_i-1}
>\binom{n}{k_{i^*}}+r-2 \\
& =\binom{n}{k_{i^*}}+r-2+\binom{n-k_r}{k_1-k_r}-
\binom{n-k_r}{k_{i^*}}=|\mathcal R_{k_r}^{(k_{i^*})}|+\sum_{i\in[r]\setminus\{i^*\}}|\mathcal P_{k_r}^{(k_i)}|,
\end{align*}
and hence $\sum_{i=1}^{r}|\mathcal F_i|\leq\sum_{i=1}^{r}\binom{n-1}{k_i-1}$ with equality if and only if  $\mathcal F_i=\mathcal P_1^{(k_i)}$ for every $i\in[r]$ up to isomorphism.

\noindent\underline{{\bf Case 3 $n>k_1+k_r$}.}
Since $k_1\geq\cdots\geq k_r$, we have $n>k_{i^*}+\overline{k_{i^*}}$ for every $i^*\in[r]$, where $\overline{k_{i^*}}=\min\{k_i:i\in[r]\setminus\{{i^*}\}\}$.
By Theorem \ref{thm:nonempty_k1_k2_kr_other} $(2)$, if the value of $\sum_{i=1}^{r}|\mathcal F_i|$ reaches the largest, then up to isomorphism, one of the following three cases holds:
\begin{enumerate}
\item[$(\alpha)$] $\mathcal F_i=\mathcal P_{1}^{(k_i)}$ for every $i\in[r]$;
\item[$(\beta)$] if $i^*\neq r$, then $\mathcal F_{i^*}=\mathcal R_{k_r}^{(k_{i^*})}$ and $\mathcal F_{i}=\mathcal P_{k_r}^{(k_i)}$ for every $i\in[r]\setminus\{i^*\}$;
\item[$(\gamma)$] if $i^*=r$, then $\mathcal F_{i^*}=\mathcal R_{k_{r-1}}^{(k_{i^*})}$ and $\mathcal F_{i}=\mathcal P_{k_{r-1}}^{(k_i)}$ for every $i\in[r-1]$.
\end{enumerate}

If $k_{r-1}=k_r=1$, then the above three cases are the same and so $\sum_{i=1}^{r}|\mathcal F_i|\leq\sum_{i=1}^{r}\binom{n-1}{k_i-1}$ with equality if and only if  $\mathcal F_i=\mathcal P_1^{(k_i)}$ for every $i\in[r]$ up to isomorphism.

If $k_{r-1}>k_r=1$, then the cross-intersecting families of type $(\alpha)$ are the same as the cross-intersecting families of type $(\beta)$. Applying Lemma \ref{lem:Comparative_size_4} (note that $n\geq k_1+k_2$) to compare the values of $\sum_{i=1}^{r}|\mathcal F_i|$ with types $(\alpha)$ and $(\gamma)$, we have $\sum_{i=1}^{r}|\mathcal P_{1}^{(k_i)}|>\sum_{i=1}^{r-1}|\mathcal P_{k_{r-1}}^{(k_i)}|+|\mathcal R_{k_{r-1}}^{(k_r)}|$.
Thus $\sum_{i=1}^{r}|\mathcal F_i|\leq\sum_{i=1}^{r}\binom{n-1}{k_i-1}$ with equality if and only if  $\mathcal F_i=\mathcal P_1^{(k_i)}$ for every $i\in[r]$ up to isomorphism.

If $k_1=k_r\geq2$, then $k_1=k_2=\dots=k_r$ and so the cross-intersecting families of type $(\beta)$ are isomorphic to the cross-intersecting families of type $(\gamma)$. Thus $\sum_{i=1}^{r}|\mathcal F_i|\leq\max\{\binom{n}{k_1}-\binom{n-k_r}{k_1}+\sum_{i=2}^{r}\binom{n-k_r}{k_i-k_r},\sum_{i=1}^{r}\binom{n-1}{k_i-1}\}$  with equality if and only if up to isomorphism, either $\mathcal F_i=\mathcal P_1^{(k_i)}$ for every $i\in[r]$, or $\mathcal F_{j}=\mathcal R_{k_r}^{(k_{j})}$ for some $j\in[r]$ and $\mathcal F_{i}=\mathcal P_{k_r}^{(k_i)}$ for every $i\in[r]\setminus\{j\}$.

Assume that $k_1>k_r\geq2$. Apply Lemma \ref{lem:Comparative_size_4} to compare the values of $\sum_{i=1}^{r}|\mathcal F_i|$ with types $(\alpha)$ and $(\gamma)$ to obtain $\sum_{i=1}^{r}|\mathcal P_{1}^{(k_i)}|>\sum_{i=1}^{r-1}|\mathcal P_{k_{r-1}}^{(k_i)}|+|\mathcal R_{k_{r-1}}^{(k_r)}|$. When $k_1>k_{i^*}$, compare the values of $\sum_{i=1}^{r}|\mathcal F_i|$ with types $(\alpha)$ and $(\beta)$ to obtain $\sum_{i=1}^{r}|\mathcal P_{1}^{(k_i)}|>\sum_{i\in[r]\setminus\{i^*\}}|\mathcal P_{k_{r}}^{(k_i)}|+|\mathcal R_{k_{r}}^{(k_{i^*})}|$. Therefore, $\sum_{i=1}^{r}|\mathcal F_i|\leq\max\{\binom{n}{k_1}-\binom{n-k_r}{k_1}+\sum_{i=2}^{r}\binom{n-k_r}{k_i-k_r},\sum_{i=1}^{r}\binom{n-1}
{k_i-1}\}$ with equality if and only if up to isomorphism, either $\mathcal F_i=\mathcal P_1^{(k_i)}$ for every $i\in[r]$, or $\mathcal F_{1}=\mathcal R_{k_r}^{(k_{1})}$ and $\mathcal F_{i}=\mathcal P_{k_r}^{(k_i)}$ for every $i\in[r]\setminus\{1\}$.
\end{proof}

\section{Concluding remarks}

For non-empty cross-intersecting families $\mathcal F_i\subseteq\binom{[n]}{k_i}$, $i\in[r]$, this paper examines the largest value of $\sum_{i=1}^{r}|\mathcal F_i|$ and determines the structure of $\mathcal F_i$'s that achieves the largest sum. Theorems \ref{thm:nonempty_k1_k2_kr_other} and \ref{thm:nonempty_k1_k2_kr_largest} generalize Theorems \ref{thm:nonempty_2_nonuniform} and \ref{thm:nonempty_r_uniform} by extending two families to arbitrary number of families allowing different sizes.

We remark that Theorem \ref{thm:nonempty_k1_k2_kr_other} is not equivalent to Theorem \ref{thm:nonempty_k1_k2_kr_largest}. For example, take $r=3$, $n=10$, $k_1=5$, $k_2=3$ and $k_3=2$. Suppose that $\mathcal F_1\subseteq\binom{[n]}{k_1},\mathcal F_2\subseteq\binom{[n]}{k_2}$ and $\mathcal F_3\subseteq\binom{[n]}{k_3}$ are non-empty cross-intersecting families such that $|\mathcal F_3|\geq\binom{n-1}{k_3-1}=9$. Apply Theorem \ref{thm:nonempty_k1_k2_kr_other} with $i^*=3$ to obtain $\sum_{i=1}^{3}|\mathcal F_i|\leq\max\{46,171\}=171$ with equality if and only if there is $x\in[n]$ such that $\mathcal F_i=\{F\in\binom{[n]}{k_i}:x\in F\}$ for $1\leq i\leq3$. Apply Theorem \ref{thm:nonempty_k1_k2_kr_largest} to obtain $\sum_{i=1}^{3}|\mathcal F_i|\leq\max\{205,171\}=205$ with equality if and only if there is $S\in\binom{[n]}{k_3}$ such that $\mathcal F_1=\{F\in\binom{[n]}{k_1}:F\cap S\neq\emptyset\},\mathcal F_2=\{F\in\binom{[n]}{k_2}:S\subseteq F\}$ and $\mathcal F_3=\{S\}$. Since it is required that $|\mathcal F_3|\geq 9$, by Theorem \ref{thm:nonempty_k1_k2_kr_largest}, we only know $\sum_{i=1}^{3}|\mathcal F_i|<205$.


\begin{thebibliography}{99}

\bibitem{Brog09}
P.~Borg, A short proof of a cross-intersection theorem of Hilton, Discrete Math., 309 (2009), 4750--4753.

\bibitem{Brog15}
P.~Borg, A cross-intersection theorem for subsets of a set, Bull. London Math. Soc., 47 (2015), 248--256.

\bibitem{EKR}
P.~Erd\H{o}s, C.~Ko, and R.~Rado, Intersection theorems for systems of finite sets, Quart. J. Math. Oxford Ser., 12 (1961), 313--320.

\bibitem{Frankl21}
P.~Frankl, Old and new applications of Katona's circle, European J. Combin., 95 (2021), 103339.

\bibitem{FT}
P.~Frankl and N.~Tokushige, Some best possible inequalities concerning cross-intersecting families, J. Combin. Theory Ser. A, 61 (1992), 87--97.

\bibitem{FG}
J.~F\"{u}redi and J.~R.~Griggs, Families of finite sets with minimum shadows, Combinatorica, 6 (1986), 355--363.

\bibitem{Hilton76}
A.~J.~W.~Hilton, The Erd\H{o}s-Ko-Rado theorem with valency conditions, Unpublished Manuscript, 1976.

\bibitem{Hilton77}
A.~J.~W.~Hilton, An intersection theorem for a collection of families of subsets of a finite set, J. London Math. Soc., 2 (1977), 369--376.

\bibitem{HM}
A.~J.~W.~Hilton and E.~C.~Milner, Some intersection theorems for systems of finite sets, Quart. J. Math. Oxford Ser., 18 (1967), 369--384.

\bibitem{Katona66}
G.~O.~H.~Katona, A theorem of finite sets, in: Theory of Graphs, Proc. Colloq. Tihany, Akad\'{e}mai Kiad\'{o}, 187--207, 1968.

\bibitem{Kruskal}
J.~B.~Kruskal, The number of simplices in a complex, in: Mathematical Optimization Techniques, University of California Press, 251--278, 1963.

\bibitem{KZ}
A.~Kupavskii and D.~Zakharov, Regular bipartite graphs and intersecting families, J. Combin. Theory Ser. A, 155 (2018), 180--189.

\bibitem{Mors}
M.~M\"{o}rs, A generalization of a theorem of Kruskal, Graphs Combin., 1 (1985), 167--183.

\bibitem{Pyber}
L.~Pyber, A new generalization of the Erd\H{o}s-Ko-Rado theorem, J. Combin. Theory Ser. A, 43 (1986), 85--90.

\bibitem{SFQ}
C.~Shi, P.~Frankl, and J.~Qian, On non-empty cross-intersecting families, Combinatorica, (2022). https://doi.org/10.1007/s00493-021-4839-4.
\end{thebibliography}
\end{document}